\documentclass{amsart}
\usepackage{graphicx}
\usepackage{verbatim}
\usepackage{textcomp}
\newtheorem{thm}{Theorem}[section]

\newtheorem{lem}[thm]{Lemma}
\newtheorem{prop}[thm]{Proposition}
\theoremstyle{definition}

\theoremstyle{remark}
\newtheorem{rem}[thm]{Remark}
\numberwithin{equation}{section}
\newcommand{\norm}[1]{\left\Vert#1\right\Vert}
\newcommand{\Real}{\mathbb R}
\newcommand{\eps}{\varepsilon}
\newcommand{\To}{\longrightarrow}
\newcommand{\dist}[0]{\mathrm{dist}}
\newcommand{\ra}{\rangle}
\newcommand{\la}{\langle}
\newcommand{\Comp}{\mathbb C}
\newcommand{\lam}{\lambda}

\title{Harmonic map heat flow with rough boundary data}
\author{Lu Wang}%
\address{Department of Mathematics\\
Massachusetts Institute of Technology\\
77 Massachusetts Avenue, Cambridge, MA 02139}
\email{luwang@math.mit.edu}

\begin{document}
\begin{abstract}
Let $B_1$ be the unit open disk in $\Real^2$ and $M$ be a closed Riemannian manifold. In this note, we first prove the uniqueness for weak solutions of the harmonic map heat flow in $H^1([0,T]\times B_1,M)$ whose energy is non-increasing in time, given initial data $u_0\in H^1(B_1,M)$ and boundary data $\gamma=u_0|_{\partial B_1}$. Previously, this uniqueness result was obtained by Rivi\`{e}re (when $M$ is the round sphere and the energy of initial data is small) and Freire (when $M$ is an arbitrary closed Riemannian manifold), given that $u_0\in H^1(B_1,M)$ and $\gamma=u_0|_{\partial B_1}\in H^{3/2}(\partial B_1)$. The point of our uniqueness result is that no boundary regularity assumption is needed. Second, we prove the exponential convergence of the harmonic map heat flow, assuming that energy is small at all times. 
\end{abstract}
\maketitle

\section{Introduction}
Let $B_1$ be the unit open disk in $\Real^2$ and $M$ be a closed Riemannian manifold. Suppose that $u\in H^1([0,T]\times B_1,M)$ is the weak solution of the initial-boundary value problem for the harmonic map heat flow, given initial data $u_0\in H^1(B_1,M)$ and boundary data $\gamma=u_0|_{\partial B_1}$\footnote{The restriction here is the trace operator acting on $u_0$.}. In this note, we study the uniqueness and the rate of convergence of the weak solution $u$.

Under the additional assumption that $\gamma\in H^{3/2}(\partial B_1,M)$, the initial-boundary value problem for the harmonic map heat flow has been investigated intensively by several mathematicians, such as Chang, Rivi\`{e}re and Freire; see \cite{Ch}, \cite{Ri}, \cite{Fr1}, \cite{Fr2} and \cite{Fr3}. Define 
\begin{equation}
\label{StruweSpace}
V^T=H^1([0,T]\times B_1,M)\cap L^\infty([0,T],H^1(B_1,M))\cap L^2([0,T],H^2(B_1,M)).
\end{equation}
Space $V^T$ plays a crucial role in the previous referred papers. 
However, without $\gamma\in H^{3/2}(\partial B_1)$, we are not able to show that $u$ with non-increasing energy is in $V^{T'}$ for some $T'>0$, as Freire did in \cite{Fr2}. Because, otherwise, it would imply that $\gamma\in H^{3/2}(\partial B_1,M)$ by the Sobolev trace theorem. But the image of trace operator on $H^1(B_1)$ is exactly $H^{1/2}(\partial B_1)$ and $H^{3/2}(\partial B_1)$ is a proper subset of $H^{1/2}(\partial B_1)$.

To get around this, we make use of the interior gradient estimate for $u$ and Hardy's inequality. The main difficulty is to deal with the $L^2$ inner product of $|\nabla u|^2$ and $h^2$ for $\forall h\in H^1_0(B_1)$, which arises from the non-linear term in the harmonic map heat flow equation. First, using H\'{e}lein's existence result of the Coulomb frame, we derive the interior gradient estimate for $u$ with small energy. Second, assuming the energy is non-increasing in time, we conclude that $u(t,\cdot)\To u_0$ in the $H^1(B_1)$ topology and thus we can uniformly bound the energy on small disks for short time. Thus, applying the gradient estimate for $u$ restricted on small disks, we bound $|\nabla u|^2(t,x)$ by $t^{-1}$ and $(1-|x|)^{-2}$ for short time. Finally, by Hardy's inequality, we bound the $L^2$ inner product of $|\nabla u|^2$ and $h^2$ for small time $t>0$.

In section 5, we prove the existence and uniqueness theorem for weak solutions of the harmonic map heat flow in $\cap_{T>0}H^1([0,T]\times B_1,M)$ whose energy is non-increasing in time, given small energy initial data $u_0\in H^1\cap C^0(\bar{B}_1)$ and boundary data $\gamma =u_0|_{\partial B_1}$. By the Sobolev trace theorem, such weak solutions are not in $V^T$ in general. 

Throughout, we use the subscripts $t$, $x_1$, $x_2$ and $r$ to denote taking derivatives with respect to $t$, $x_1$, $x_2$ and $r$; $\nabla\cdot$ and $\nabla^2\cdot$ denote the gradient and the Hessian operator respectively; \textgravedbl$\sup$\textacutedbl in this note is \textgravedbl esssup\textacutedbl in the usual literature; constants in proofs are not preserved when crossing lemmas, propositions and theorems. 

By the Nash embedding theorem, $M$ can be isometrically embedded in some Euclidean space $(\Real^N,\la,\ra)$. Given $w\in H^1(B_1,M)$, we define the energy functional $E(w)=\frac{1}{2}\int_{B_1}|\nabla w|^2$. The harmonic map heat flow is the negative $L^2$ gradient flow of the energy functional. Thus, given $u_0\in H^1(B_1,M)$ and $\gamma=u_0|_{\partial B_1}$, $u\in H^1([0,T]\times B_1,M)$\footnote{By Theorem 3 on page 287 of \cite{E}, if $w\in H^1([0,T]\times B_1)$, then there exists $\tilde{w}\in C^0([0,T],L^2(B_1))$ and $\tilde{w}(t,\cdot)=w(t,\cdot)$ for $a.e.$ $t$. Thus, in this note, we always choose $\tilde{w}$ representing $w$ in $H^1([0,T]\times B_1)$. In other words, we always assume that functions in $H^1([0,T]\times B_1)$ are actually also in $C^0([0,T],L^2(B_1))$.} is the weak solution of the initial-boundary value problem for the harmonic map heat flow, if
\begin{eqnarray}
\label{HMHFEqn}
\left\{\begin{array}{lll}
u_t-\Delta u=-A_u(\nabla u,\nabla u)\quad & \text{on}\quad (0,T)\times B_1\\
u(t,x)=\gamma(x)\quad & \text{for}\quad t\ge 0, x\in \partial B_1\\
\lim_{t\To 0+}u(t,\cdot)=u_0\quad & \text{in}\quad L^2(B_1,M)\quad\text{topology}
\end{array}
\right.
\end{eqnarray}
where $A$ is the second fundamental form of $M$ in $\Real^N$ at the point $u$. We recall that
\begin{equation}
\label{GeoFact}
A_u(\nabla u,\nabla u)\ \text{is perpendicular to}\ M\ \text{at}\ u.
\end{equation}
We define the weak solution $u\in H^1([0,T]\times B_1,M)$ of the first equation in (\ref{HMHFEqn}) by 
\begin{equation}
\label{HMHFEqn1}
\int_0^T\int_{B_1}\la u_t,\xi \ra +\la\nabla u,\nabla\xi \ra +\la A_u(\nabla u,\nabla u),\xi\ra dxdt=0.,
\end{equation}
for $\forall \xi\in C^\infty_c((0,T)\times B_1,\Real^N)$. In \cite{W}, it is shown that this definition is equivalent to that, for $a.e.$ $t\in [0,T]$,
\begin{equation}
\label{HMHFEqn2}
\int_{\{t\}\times B_1}\la u_t,\zeta \ra +\la\nabla u,\nabla\zeta \ra +\la A_u(\nabla u,\nabla u),\zeta\ra dx=0,\ \ \forall \zeta\in H^1_0\cap L^\infty(B_1,\Real^N).
\end{equation}
Note that the definition given by equation (\ref{HMHFEqn2}) allows us to freeze the time and is more convenient for our proofs of Theorems \ref{1stThm} and \ref{2ndThm}.

The main results of this note are two folds. First, we show the uniqueness for weak solutions of (\ref{HMHFEqn}) whose energy is non-increasing. 
\begin{thm}
\label{1stThm}
If $u$ and $v$ are weak solutions of (\ref{HMHFEqn}) in $H^1([0,T]\times B_1,M)$ satisfying $E(u(t_2,\cdot))\leq E(u(t_1,\cdot))$, $E(v(t_2,\cdot))\leq E(v(t_1,\cdot))$ for $t_1\leq t_2$, and having the same initial data $u_0\in H^1(B_1,M)$ and boundary value $\gamma=u_0|_{\partial B_1}$, then $u=v$ on $[0,T]\times B_1$.
\end{thm}
\begin{rem}
\label{1stRem}
In \cite{BDV}, Bertsch, Dal Passo and van der Hout proved that there exist initial data $u_0\in H^1(B_1,S^2)$ and boundary data $\gamma=u_0|_{\partial B_1}$ such that (\ref{HMHFEqn}) has infinitely many weak solutions which do not satisfy the non-increasing energy condition. Thus, Theorem \ref{1stThm} appears to be the optimal uniqueness statement for weak solutions of the harmonic map heat flow with time independent boundary data. 
\end{rem}
Second, we study the rate of convergence of small energy weak solutions of (\ref{HMHFEqn}). And we conclude that
\begin{thm}
\label{2ndThm}
There exists $\eps_0>0$, depending only on $M$, so that: if $u$ is a weak solution of (\ref{HMHFEqn}) in $\cap_{T>0}H^1([0,T]\times B_1,M)$ satisfying that $E(u(t,\cdot))<\eps_0$ for $a.e.\ t$, then there exist $T_0>0$, $\alpha_0>0$ and $C_0>0$ such that, for $a.e \ t>T_0$,
\begin{equation}
\label{ConvEqn}
\norm{u(t,\cdot)-u_\infty}_{H^1}\le C_0\text{e}^{-\alpha_0 t},
\end{equation}
where $u_\infty$ is some harmonic map from $B_1$ to $M$ with the same boundary value $\gamma$.
\end{thm}
\begin{rem}
\label{3rdRem}
It follows from Corollary 3.3 in \cite{CM2} that $u_\infty$ is the unique harmonic map in the class of 
\begin{equation*}
\{w\in H^1(B_1,M)\ \vline\  w|_{\partial B_1}=\gamma\ \text{and}\ E(w)\le E_1\},
\end{equation*}
where $E_1>0$ is a constant depending only on $M$.
\end{rem}
\begin{rem}
\label{2ndRem}
Although Theorem \ref{1stThm} and \ref{2ndThm} are stated for the unit open disk in this note, the proofs could be modified to apply to any bounded open set in $\Real^2$ and even general two dimensional Riemannian manifolds. 
\end{rem}
\textbf{Acknowledgement.} The author would like to thank Prof. Tobias Colding for suggesting this question and his continuous guidance. Also, the author is grateful to Jacob Bernstein and the anonymous referees for their useful comments.

\section{Interior gradient estimate for the harmonic map heat flow}
In this section, we derive the interior gradient estimate for small energy solutions of the harmonic map heat flow. This is one of the key ingredients in the proofs of Theorem \ref{1stThm} and Theorem \ref{2ndThm}. First, using H\'{e}lein's existence result of the Coulomb frame, we show that $u(t,\cdot)\in H^2(B_{1/2},M)$ for $a.e.\ t$. Next, we follow Struwe's method in \cite{St1} to conclude that $u\in L^2([0,\bar{T}],H^2(B_{1/4},M))$ for $1<\bar{T}<2$ and obtain the gradient estimate for $u$ at $(1,0,0)$.

The following elementary geometric fact is obtained in \cite{CM1} and will be used frequently in this note. For self-containedness, we include the proof in Appendix A. 
\begin{lem}
\label{EleGeo}
(Lemma A.1 in \cite{CM1})
There exists $C>0$, depending on $M$, so that: if $x,y\in M$, then $|(x-y)^{\perp}|\le C|x-y|^2$, where $(x-y)^{\perp}$ is the normal component to $M$ at $y$.
\end{lem} 
First, we derive the local integral bounds for $|\nabla^2u|$ and $|\nabla u|$:
\begin{lem}
\label{5thLem}
Suppose that $1<\bar{T}<2$ and $u\in H^1([0,\bar{T}]\times B_1,M)$ satisfies 
\begin{equation}
\label{2ndEqn}
u_t-\Delta u=-A_u(\nabla u,\nabla u)
\end{equation}
on $(0,\bar{T})\times B_1$. Then there exists $\eps>0$, depending only on $M$, so that: if $E(u(t,\cdot))\le \eps$ for $a.e.$ $t\in [0,\bar{T}]$, then
\begin{eqnarray}
\label{26thEqn}
&&\int_0^{\bar{T}}\int_{B_{1/4}}|\nabla^2u|^2dxdt\le 10^5\sup_{0\le t\le\bar{T}}E(u(t,\cdot)),\\
&&\int_0^{\bar{T}}\int_{B_{1/4}}|\nabla u|^4dxdt\le 10^8\sup_{0\le t\le\bar{T}}E(u(t,\cdot))^2.
\end{eqnarray}
\end{lem}
\begin{proof}
First, note that $|u_t|\in L^2(B_1)$ for $a.e.\ t$. Fix such $t$. Following the proof of Theorem 4.1.1 in \cite{He}, there exists $\delta_1>0$, depending only on $M$, so that: if $E(u(t,\cdot))\leq \delta_1$, then there is a finite energy harmonic section (so called \textgravedbl Coulomb frame\textacutedbl) $e(t)=(e_1(t),\dots,e_n(t))$ of the bundle of orthonormal frames for $u(t,\cdot)^*(TM)$, and one can construct $\beta(t)\in L^\infty(B_1,GL(n,\Comp))$ satisfying that $|\beta(t)|\leq \lam_1$, $|\beta(t)^{-1}|\leq \lam_1$, and 
\begin{equation}
\partial_{\bar{z}}(\beta^{-1}\alpha(t))
=\frac{1}{4}\beta(t)^{-1}f,
\end{equation}
where $\lam_1$ depends only on $M$ and the upper bound of the energy of $u(t,\cdot)$, $z=x_1+ix_2$, $\alpha=(\la \partial_z u,e_1\ra,\dots,\la \partial_z u,e_n\ra)$  and $f=(\la u_t,e_1\ra,\dots,\la u_t,e_n\ra)$. Thus, by the elliptic regularity for $\partial_{\bar{z}}$ operator (see the theorem on page 80 of \cite{Ha}), $\beta^{-1}\alpha(t)\in H^1(B_{3/4})$. It follows from the Sobolev embedding theorem (see Theorem 2 on page 265 of \cite{E}) and $|\beta(t)|\leq \lam_1$ that $\alpha(t)\in L^p(B_{3/4})$ for $1<p<\infty$. In particular, $|\nabla u|(t,\cdot)\in L^4(B_{3/4})$. Therefore, by Theorem 8.8 in \cite{GT}, $u(t,\cdot)\in H^2(B_{1/2})$. 

Next, let $\phi$ be a smooth cut-off function, which is one in $B_{1/4}$, compactly supported in $B_{1/2}$, $0\leq\phi\leq 1$ and $|\nabla\phi|\leq 8$. Thus, by Lemma 6.7 in Chapter III of \cite{St2} and equation (\ref{2ndEqn}), 
\begin{eqnarray*}
&&\int_{\{t\}\times B_1}|\Delta u|^2\phi^2dx\le\int_{\{t\}\times B_1}|u_t|^2\phi^2dx+\sup_{M}|A|^2\int_{\{t\}\times B_1}|\nabla u|^4\phi^2\\
&\leq&\int_{\{t\}\times B_1}|u_t|^2\phi^2dx
+\lam_2 E(u(t,\cdot))\left(\int_{\{t\}\times B_1}|\nabla^2 u|^2\phi^2dx+\int_{\{t\}\times B_1}|\nabla u|^2dx\right),
\end{eqnarray*}
where $\lam_2=512\sup_M|A|^2$. On the other hand, approximating $u(t,\cdot)$ by smooth functions in $H^2(B_{1/2})$ and integration by parts, give
\begin{equation}
\int_{\{t\}\times B_1}|\Delta u|^2\phi^2dx\geq \frac{1}{2}\int_{\{t\}\times B_1}|\nabla^2u|^2\phi^2dx-8\int_{\{t\}\times B_1}|\nabla u|^2|\nabla \phi|^2dx.
\end{equation}
If $4\lam_2 E(u(t,\cdot))\leq 1$, then 
\begin{equation}
\label{11thEqn}
\int_{\{t\}\times B_1}|\nabla^2 u|^2\phi^2dx\leq 32\int_{\{t\}\times B_1}\left[|u_t|^2\phi^2+|\nabla u|^2(1+|\nabla\phi|^2)\right]dx.
\end{equation}
Thus, integrating over $[0,\bar{T}]$, we have
\begin{equation}
\int_0^{\bar{T}}\int_{B_1}|\nabla^2u|^2\phi^2dxdt\leq 32\int_0^{\bar{T}}\int_{B_1}\left[|u_t|^2\phi^2+|\nabla u|^2(1+|\nabla\phi|^2)\right]dxdt,
\end{equation}
if $E(u(t,\cdot))\leq \min\{\delta_1,\lam_2^{-1}/4\}$ for $a.e.\ t$. And it follows from the proof of Lemma 3.4 in \cite{St1} (replacing the test function $u$ by $u\phi^2$) that
\begin{equation}
\label{10thEqn}
\int_0^{\bar{T}}\int_{B_1}|u_t|^2\phi^2dxdt\leq 1026\sup_{0\leq t\leq \bar{T}}E(u(t,\cdot)).
\end{equation}
Therefore,
\begin{equation}
\label{5thEqn}
\int_0^{\bar{T}}\int_{B_1}|\nabla^2u|^2\phi^2dxdt\leq 10^5\sup_{0\leq t\leq \bar{T}}E(u(t,\cdot)),
\end{equation}
and it follows from Lemma 6.7 in Chapter III of \cite{St2} that
\begin{eqnarray*}
\int_0^{\bar{T}}\int_{B_1}|\nabla u|^4\phi^4dxdt &\le& 8\sup_{0\leq t\leq \bar{T}}E(u(t,\cdot))\int_0^{\bar{T}}\int_{B_1}\left(|\nabla\phi|^2|\nabla u|^2+|\nabla^2 u|^2\phi^2\right)dxdt\\
&\le& 10^8\sup_{0\le t\le\bar{T}}E(u(t,\cdot))^2.
\end{eqnarray*}
\end{proof}
Now we are ready to prove the interior gradient estimate:
\begin{lem}
\label{1stLem}
Under the assumption of Lemma \ref{5thLem}, there exist $\epsilon_1\in (0,\eps]$ and $C_1>0$, depending only on $M$, so that: if $E(u(t,\cdot))\leq \epsilon_1$ for $a.e.\ t\in[0,\bar{T}]$,
then 
\begin{equation}
\label{gradientEqn}
|\nabla u|^2(1,0,0)\leq C_1\sup_{0\le t\le\bar{T}}E(u(t,\cdot)).
\end{equation}
\end{lem}
\begin{proof}
We will follow the suggestion in the remark after Lemma 3.10 in \cite{St1} to obtain the interior gradient estimate for $u$. Let $\phi$ be a smooth cut-off function, which is one in $B_{1/8}$, compactly supported in $B_{1/4}$, $0\le \phi\le 1$ and $|\nabla \phi|\le 16$.
Also, we define $D^hw(t,x)=(w(t+h,x)-w(t,x))/h$ for $0<h<h_0\ll 1$, where $w$ takes value in $\Real$ or $\Real^N$. Thus, for $0<t_1\leq t_2\leq \bar{T}-h_0$, using equation (\ref{2ndEqn}) and integration by parts, we get
\begin{eqnarray*}
&&\int_{t_1}^{t_2}\int_{B_1}\partial_t|D^h u|^2\phi^2dxdt
+2\int_{t_1}^{t_2}\int_{B_1}|\nabla D^hu|^2\phi^2dxdt\\
&\le& 4\int_{t_1}^{t_2}\int_{B_1}|D^hu||\phi||\nabla D^hu||\nabla\phi|dxdt+2h^{-1}\int_{t_1}^{t_2}\int_{B_1}\la A_{u}(\nabla u,\nabla u),D^hu\ra\phi^2 dxdt\\
&&-2h^{-1}\int_{t_1}^{t_2}\int_{B_1}\la A_{u(t+h,x)}(\nabla u,\nabla u),D^hu\ra\phi^2dxdt\\
&\le& 4\int_{t_1}^{t_2}\int_{B_1}|D^hu||\phi||\nabla D^hu||\nabla\phi|dxdt+\lam_1\int_{t_1}^{t_2}\int_{B_1}
|D^hu|^2|\nabla u|^2(t,x)\phi^2dxdt\\
&&+\lam_1\int_{t_1}^{t_2}\int_{B_1}
|D^hu|^2|\nabla u|^2(t+h,x)\phi^2dxdt\\
&\le& \int_{t_1}^{t_2}\int_{B_1}|\nabla D^hu|^2\phi^2dxdt+4\int_{t_1}^{t_2}\int_{B_1}
|D^hu|^2|\nabla\phi|^2dxdt\\
&& +\lam_1\int_{t_1}^{t_2}\int_{B_1}|D^hu|^2
\left(|\nabla u|^2(t,x)+|\nabla u|^2(t+h,x)\right)\phi^2dxdt,
\end{eqnarray*}
where $\lam_1=C\sup_M|A|$, and we use (\ref{GeoFact}) and Lemma \ref{EleGeo} in the second inequality. Thus, absorbing the first term on the right hand in the left, gives
\begin{equation}
\label{13rdEqn}
\begin{split}
&\quad \int_{\{t_2\}\times B_1}|D^hu|^2\phi^2dx+\int_{t_1}^{t_2}\int_{B_1}|\nabla D^hu|^2\phi^2dxdt\\
&\le \int_{\{t_1\}\times B_1}|D^hu|^2\phi^2dx+4\int_{t_1}^{t_2}\int_{B_1}|D^hu|^2
|\nabla\phi|^2dxdt\\
&\quad +\lam_1\int_{t_1}^{t_2}\int_{B_1}|D^hu|^2
\left(|\nabla u|^2(t,x)+|\nabla u|^2(t+h,x)\right)\phi^2dxdt.
\end{split}
\end{equation}
By Lemma 6.7 in Chapter III of \cite{St2}, we get
\begin{eqnarray*}
&&\int_{t_1}^{t_2}\int_{B_1}|D^hu|^4\phi^4dxdt\\
&\le& 8\sup_{t_1\le t\le t_2}\int_{\{t\}\times B_1}|D^hu|^2\phi^2dx\cdot
\int_{t_1}^{t_2}\int_{B_1}\left(|D^hu|^2|\nabla\phi|^2+|\nabla D^hu|^2\phi^2\right)dxdt.
\end{eqnarray*}
If $8\lam_1\sup_{0\le t\le \bar{T}}E(u(t,\cdot))<10^{-4}$, then, by Lemma \ref{5thLem} and H\"{o}lder's inequality, we have
\begin{eqnarray*}
&&\lam_1\int_{t_1}^{t_2}\int_{B_1}|D^hu|^2\left(|\nabla u|^2(t,x)+|\nabla u|^2(t+h,x)\right)\phi^2dxdt\\
&\leq& 2\lam_1\left(\int_{t_1}^{t_2}\int_{B_1}|D^hu|^4\phi^4dxdt\right)^{\frac{1}{2}}
\left(\int_{t_1}^{t_2}\int_{B_{1/4}}\left(|\nabla u|^4(t,x)+|\nabla u|^4(t+h,x)\right)dxdt\right)^{\frac{1}{2}}\\
&\le& \left(\sup_{t_1\le t\le t_2}\int_{\{t\}\times B_1}|D^hu|^2\phi^2dx\right)^{\frac{1}{2}}
\left(\int_{t_1}^{t_2}\int_{B_1}\left(|D^hu|^2|\nabla\phi|^2
+|\nabla D^hu|^2\phi^2\right)dxdt\right)^{\frac{1}{2}}\\
&\le& \frac{1}{2}\sup_{t_1\le t\le t_2}\int_{\{t\}\times B_1}|D^hu|^2\phi^2dx
+\frac{1}{2}\int_{t_1}^{t_2}\int_{B_1}\left(|D^hu|^2|\nabla\phi|^2
+|\nabla D^hu|^2\phi^2\right)dxdt.
\end{eqnarray*}
Hence, (\ref{13rdEqn}) gives
\begin{equation}
\label{22thEqn}
\begin{split}
&\quad\quad \int_{\{t_2\}\times B_1}|D^hu|^2\phi^2dx-\int_{\{t_1\}\times B_1}|D^hu|^2\phi^2dx\\
&\le\quad \frac{9}{2}\int_{t_1}^{t_2}\int_{B_1}|D^hu|^2
|\nabla\phi|^2dxdt+\frac{1}{2}\sup_{t_1\le t\le t_2}\int_{\{t\}\times B_1}|D^hu|^2\phi^2dx.
\end{split}
\end{equation}
We conclude from (\ref{22thEqn}) that
\begin{eqnarray*}
\int_{\{t_2\}\times B_1}|D^hu|^2\phi^2dx
&\le& 2\inf_{0\le t\le t_2}\int_{\{t\}\times B_1}|D^hu|^2\phi^2dx+\lam_2\int_0^{t_2}\int_{B_{1/4}}|D^hu|^2dxdt,
\end{eqnarray*}
where $\lam_2>0$ is a universal constant. Therefore,
\begin{equation}
\label{7thEqn}
\int_{\{t_2\}\times B_1}|D^hu|^2\phi^2dx\leq 2(t_2^{-1}+\lam_2)\int_0^{t_2}\int_{B_{1/4}}|D^hu|^2dxdt.
\end{equation} 
Note that $u$ is in $H^1([0,\bar{T}]\times B_1)$ and thus
\begin{equation}
\int_0^{\bar{T}-h_0}\int_{B_1}|D^hu-u_t|^2dxdt\le\int_0^1\int_0^{\bar{T}-h_0}\int_{B_1}
|u_t(t+sh,x)-u_t(t,x)|^2 dxdtds.
\end{equation}
Since elements in $L^2$ are continuous in the mean, thus
\begin{equation}
\label{8thEqn}
\lim_{h\To 0}\int_{0}^{\bar{T}-h_0}\int_{B_1}|D^hu-u_t|^2dxdt=0.
\end{equation}
Therefore, letting $h\To 0$, by the arbitrariness of $h_0$, (\ref{7thEqn}) and (\ref{10thEqn}), we conclude that for $a.e.\ t$,
\begin{equation}
\label{9thEqn}
\int_{\{t\}\times B_1}|u_t|^2\phi^2dx\leq 10^4(t^{-1}+\lam_2)\sup_{0\le s\le \bar{T}}E(u(s,\cdot)).
\end{equation}
By the same argument used to derive (\ref{11thEqn}), for $\frac{1}{2}<t<\bar{T}$,
\begin{equation}
\int_{\{t\}\times B_1}|\nabla^2u|^2\phi^2dx\leq \lam_3\sup_{0\le s\le \bar{T}}E(u(s,\cdot))
\end{equation}
where $\lam_3>0$ is a universal constant. Hence, by the Sobolev embedding theorem, for $1<p<\infty$,
\begin{equation}
\int_{\frac{1}{2}}^{\bar{T}}\int_{B_{1/8}}|\nabla u|^pdxdt\leq \lam_4\sup_{0\le t\le \bar{T}}E(u(t,\cdot))^{\frac{p}{2}},
\end{equation}
where $\lam_4$ depends only on $p$. Thus, by inserting cut-off functions and the theorem on page 72 of \cite{Ha}, $|u_t|$ and $|\nabla^2u|$ are in $L^p([1/4,\bar{T}_1]\times B_{1/16})$ for $1<\bar{T}_1<\bar{T}$. Furthermore, using the Bochner formula and the Gauss equation, one can derive the evolution equation for $g=|\nabla u|^2$ (see page 128 of \cite{Ha}), that is,
\begin{equation}
\label{GradientEqn}
g_t-\Delta g= -2|\text{Hess}_u|^2+2\la A_u(u_{x_1},u_{x_1}),A_u(u_{x_2},u_{x_2})\ra
-2|A_u(u_{x_1},u_{x_2})|^2.
\end{equation} 
Thus, $g\in W^{1,p}([1/8,\bar{T}_2]\times B_{1/32})$ and $|\nabla^2g|\in L^p([1/8,\bar{T}_2]\times B_{1/32})$ for $1<\bar{T}_2<\bar{T}_1$.
Therefore, by the local maximum principle (see Theorem 7.36 in \cite{Lie}), 
\begin{equation}
g(1,0,0)=|\nabla u|^2(1,0,0)\le \lam_5\sup_{0\le t\le \bar{T}}E(u(t,\cdot)),
\end{equation}
where $\lam_5>0$ depends only on $M$, assuming that $E(u(t,\cdot))<\epsilon_1$ for $a.e.\ t$ and $\eps_1=\min\{\eps,10^{-4}\lam_1^{-1}/8\}$.
\end{proof}

\section{Proof of Theorem \ref{1stThm}}
In \cite{Fr2} and \cite{Fr3}, Freire first constructed the optimal tangent frames for each fixed time and rewrote equation (\ref{2ndEqn}) under these frames. Next, he used the parabolic perturbation argument to show that any weak solution $u\in H^1([0,T]\times B_1,M)$ of (\ref{HMHFEqn}), satisfying that $E(u(t,\cdot))\le E(u_0)$ for $a.e.$ $t\in [0,T]$, is in $V^{T'}$ for some $T'\in (0,T)$, given initial data $u_0\in H^1(B_1,M)$ and boundary data $\gamma=u_0|_{\partial B_1}\in H^{3/2}(B_1,M)$; see Theorem 1.1 in \cite{Fr2}. Finally, combining with the results of Struwe and Chang in \cite{St1} and \cite{Ch}, he concluded the uniqueness for energy non-increasing weak solutions of (\ref{HMHFEqn}) on $[0,T]\times B_1$ by iteration. However, without the assumption on boundary regularity, the proof by Freire does not apply to our case (as explained in the second paragraph of Introduction). Instead, we make use of the interior gradient estimate and Hardy's inequality to show the uniqueness for energy non-increasing weak solutions of (\ref{HMHFEqn}) in $H^1([0,T]\times B_1,M)$ with initial data $u_0\in H^1(B_1,M)$ and boundary data $\gamma=u_0|_{\partial B_1}$. 

We start with showing Hardy's inequality for the unit open disk. This turns out to be the other key ingredient. Such Hardy's inequality also holds for general domains in $\Real^2$; see \cite{Ne}.
\begin{lem}
\label{2ndLem}
For $h\in H^1_0(B_1,\Real)$,
\begin{equation}
\label{HardyEqn}
\int_{B_1}\frac{h^2}{(1-\sqrt{x_1^2+x_2^2})^2}dx\leq 4\int_{B_1}|\nabla h|^2dx.
\end{equation}
\end{lem}
\begin{proof}
First, we prove the lemma for $h\in C^\infty_c(B_1,\Real)$. Rewriting the left hand of inequality (\ref{HardyEqn}) in polar coordinates and using integration by parts, we get
\begin{eqnarray*}
&&\int_{B_1}\frac{h^2}{(1-\sqrt{x_1^2+x_2^2})^2}dx
=\int_0^1\int_0^{2\pi}\frac{h^2r}{(1-r)^2}d\theta dr\\
&=&-\int_0^1\int_0^{2\pi}\frac{h^2}{1-r}d\theta dr-\int_0^1\int_0^{2\pi}\frac{2hh_rr}{1-r}d\theta dr\\
&\le& 2\left(\int_0^1\int_0^{2\pi}\frac{h^2r}{(1-r)^2}d\theta dr\right)^{\frac{1}{2}}
\left(\int_0^1\int_0^{2\pi}h_r^2rd\theta dr\right)^{\frac{1}{2}}\\
&\le& 2\left(\int_{B_1}\frac{h^2}{(1-\sqrt{x_1^2+x_2^2})^2}dx\right)^{\frac{1}{2}}
\left(\int_{B_1}|\nabla h|^2dx\right)^{\frac{1}{2}}.
\end{eqnarray*}
Thus, inequality (\ref{HardyEqn}) follows by absorbing the second term of the product on the right hand side in the left. 

Since $C^\infty_c(B_1)$ is dense in $H^1_0(B_1)$, there exists a sequence of $h_n\in C^\infty_c(B_1,\Real)$ such that $h_n\To h$ in $H^1(B_1)$ topology and $h_n\To h$ $a.e.$ in $B_1$. By Fatou's Lemma (see Theorem 3 on page 648 of \cite{E}),
\begin{eqnarray*}
\int_{B_1}\frac{h^2}{(1-\sqrt{x_1^2+x_2^2})^2}dx&\le&\liminf_{n\To\infty}
\int_{B_1}\frac{h_n^2}{(1-\sqrt{x_1^2+x_2^2})^2}dx\\
&\le&\liminf_{n\To\infty}
4\int_{B_1}|\nabla h_n|^2dx=4\int_{B_1}|\nabla h|^2dx.
\end{eqnarray*}
\end{proof}
Next, to avoid repeating computation in section 5, we will prove a general stability lemma below, i.e. Lemma \ref{3rdLem}. Suppose that $u$ and $v$ are weak solutions of (\ref{HMHFEqn}) in $H^1([0,T]\times B_1,M)$ satisfying that the energy is non-increasing and with initial data $u_0$ and $v_0$ respectively. For this moment, $u_0$ may not be equal to $v_0$. And let $\eps_2=\min\{\eps_1,C^{-1}C_1^{-1}\sup_M|A|/32\}$. 

The key to the proof of Lemma \ref{3rdLem} is to bound the $L^2$ inner products $\la|\nabla u|^2,h^2\ra_{L^2}$ and $\la|\nabla v|^2,h^2\ra_{L^2}$ on $B_1$, for $\forall h\in H^1_0(B_1)$. Such integrals arise from the non-linear terms $A_u(\nabla u,\nabla u)$ and $A_v(\nabla v,\nabla v)$ in equation (\ref{2ndEqn}). First, by the energy non-increasing assumption and Lemma \ref{1stLem}, we can bound $|\nabla u|$ and $|\nabla v|$ for $x_0\in B_1$ and small time $t_0>0$. Namely, since the energy of $u(t,\cdot)$ is non-increasing in time, $u(t,\cdot)\To u_0$ weakly in $H^1(B_1)$ and strongly in $L^2(B_1)$, as $t\To 0$. Then,
\begin{eqnarray*}
&&\lim_{t\To 0}\int_{B_1}|\nabla u(t,x)-\nabla u_0|^2dx\\
&=&\lim_{t\To 0}\int_{B_1}|\nabla u(t,x)|^2dx-\int_{B_1}2\la \nabla u(t,x),\nabla u_0\ra dx+\int_{B_1}|\nabla u_0|^2dx\\
&\le&0.
\end{eqnarray*}
Thus, $u(t,\cdot)\To u_0$ strongly in $H^1(B_1)$, and by the same argument, $v(t,\cdot)\To v_0$ strongly in $H^1(B_1)$, as $t\To 0$.
Hence, by the absolute continuity of integration, there exist $R_0>0$ and $T'\in (0,\min\{R_0^2,T\}]$  so that, for $x_0\in B_1$ and $t\in [0,T']$,
\begin{equation}
\frac{1}{2}\int_{\{t\}\times (B_{R_0}(x_0)\cap B_1)}|\nabla u|^2dx<\eps_2 \quad\text{and}\quad \frac{1}{2}\int_{\{t\}\times (B_{R_0}(x_0)\cap B_1)}|\nabla v|^2dx<\eps_2.
\end{equation}
Note that equation (\ref{2ndEqn}) is invariant under the transformation $(t,x)\To(\lam^2 t,\lam x)$ for $\lam>0$, and the energy is invariant under conformal transformations of domains in $\Real^2$. Fix $(t_0,x_0)\in (0,T')\times B_1$. Let $\lam=\min\{\sqrt{t_0},1-|x_0|\}$. Define $u_{\lam}(s,y)=u(\lam^2s,x_0+\lam y)$ and $v_{\lam}(s,y)=v(\lam^2s,x_0+\lam y)$. Thus, $u_{\lam}$ and $v_{\lam}$ satisfy equation (\ref{2ndEqn}) on $(0,\lam^{-2}T')\times B_1$, and for $s\in [0,\lam^{-2}T']$, $E(u_{\lam}(s,\cdot))<\eps_2$ and $E(v_{\lam}(s,\cdot))<\eps_2$. Hence, by Lemma \ref{1stLem}, for $(t_0,x_0)\in (0,T')\times B_1$,
\begin{equation}
\label{25thEqn} 
|\nabla u_{\lam}|^2(\lam^{-2}t_0,0,0)\le C_1\eps_2 \quad\text{and}\quad
|\nabla v_{\lam}|^2(\lam^{-2}t_0,0,0)\le C_1\eps_2.
\end{equation}
Therefore,
\begin{eqnarray}
\label{12thEqn}
|\nabla u|^2(t_0,x_0)&\le &C_1[t_0^{-1}+(1-|x_0|)^{-2}]\eps_2\\
\label{19thEqn}
|\nabla v|^2(t_0,x_0)&\le &C_1[t_0^{-1}+(1-|x_0|)^{-2}]\eps_2.
\end{eqnarray}
Then, combining inequalities (\ref{12thEqn}) and (\ref{19thEqn}) with Lemma \ref{2ndLem}, we can bound the $L^2$ inner products $\la|\nabla u|^2,h^2\ra_{L^2}$ and $\la|\nabla v|^2,h^2\ra_{L^2}$ on $B_1$, for $\forall h\in H^1_0(B_1)$ and $t\in (0,T')$. 
\begin{lem}
\label{3rdLem}
There exists $C_2>0$, depending only on $M$, so that:
\begin{equation}
\int_0^{T'}\int_{B_1}|\nabla u-\nabla v|^2t^{-\frac{1}{2}}dxdt+\frac{1}{2\sqrt{T'}}\int_{\{T'\}\times B_1}|u-v|^2dx\le N,
\end{equation}
where 
\begin{equation}
\begin{split}
N&=\left(\frac{1}{\sqrt{T'}}+2\sqrt{T'}\right)\int_{B_1}
\left(|w_0|^2 +|\nabla w_0|^2\right)dx\\
&\quad +8\sqrt{2T'(E(u_0)+E(v_0))}\left(\int_{B_1}|\nabla w_0|^2dx\right)^{\frac{1}{2}}\\
&\quad +4C_2\int_0^{T'}\int_{B_1}(|\nabla u|^2+|\nabla v|^2)(|w_0|^2+|w_0|)t^{-\frac{1}{2}}dxdt.
\end{split}
\end{equation}
\end{lem}
\begin{proof}
Define $w=u-v$. It is clear that
\begin{equation}
\label{14thEqn}
\begin{split}
&\quad\ \int_0^{T'}\int_{B_1}|\nabla w|^2t^{-\frac{1}{2}}dxdt\\
&=\int_0^{T'}\int_{B_1}\la \nabla u,\nabla w\ra t^{-\frac{1}{2}}dxdt-\int_0^{T'}\int_{B_1}\la \nabla v,\nabla w\ra t^{-\frac{1}{2}}dxdt.
\end{split}
\end{equation}

We will estimate the first term of (\ref{14thEqn}) and the second term can be estimated similarly. First, by footnote 2, $w-w_0\in C^0([0,T],L^2(B_1))$ and the map $t\To \norm{w(t,\cdot)-w_0}_{L^2(B_1)}^2$ is absolutely continuous, with
\begin{eqnarray*}
\frac{d}{dt}\norm{w(t,\cdot)-w_0}_{L^2(B_1)}^2
&=&2\int_{\{t\}\times B_1}\la w_t,w-w_0\ra dx\\
&\le&2\left(\int_{\{t\}\times B_1}|w_t|^2dx\right)^{\frac{1}{2}}
\left(\int_{\{t\}\times B_1}|w-w_0|^2dx\right)^{\frac{1}{2}}
\end{eqnarray*}
for $a.e.\ t\in [0,T]$. Thus, for $\forall t_0\ge 0$, integrating over $[0,t_0]$ and by H\"{o}lder's inequality, we conclude that
\begin{eqnarray*}
\norm{w(t_0,\cdot)-w_0}_{L^2(B_1)}&\le&\int_0^{t_0}
\left(\int_{\{t\}\times B_1}|w_t|^2dx\right)^{\frac{1}{2}}dt\\
&\le&\sqrt{t_0}\left(\int_0^{t_0}\int_{B_1}|w_t|^2dxdt\right)^{\frac{1}{2}}.
\end{eqnarray*}
Therefore, 
\begin{eqnarray}
\label{17thEqn}
&&\int_0^{T'}\int_{B_1}|w-w_0|^2t^{-\frac{3}{2}}dxdt\le 2\sqrt{T'}\int_0^{T'}\int_{B_1}|w_t|^2dxdt<+\infty,\\
\label{18thEqn}
&&\lim_{t\To 0}t^{-\frac{1}{2}}\int_{\{t\}\times B_1}|w-w_0|^2dx=0.
\end{eqnarray} 
Next, (\ref{HMHFEqn2}) gives,
\begin{eqnarray*}
&&\int_0^{T'}\int_{B_1}\la \nabla u,\nabla w\ra t^{-\frac{1}{2}}dxdt\\
&=&\int_0^{T'}\int_{B_1}\la \nabla u, \nabla w-\nabla w_0\ra t^{-\frac{1}{2}}dxdt
+\int_0^{T'}\int_{B_1}\la \nabla u, \nabla w_0\ra t^{-\frac{1}{2}}dxdt\\
&=&\int_0^{T'}\int_{B_1}\la -u_t,w-w_0\ra t^{-\frac{1}{2}}dxdt+\int_0^{T'}\int_{B_1}\la \nabla u,\nabla w_0\ra t^{-\frac{1}{2}}dxdt\\
&&-\int_0^{T'}\int_{B_1}\la A_u(\nabla u,\nabla u),w-w_0\ra t^{-\frac{1}{2}}dxdt\\
&\le&\int_0^{T'}\int_{B_1}\la -u_t,w-w_0\ra t^{-\frac{1}{2}}dxdt+2\sqrt{2T'E(u_0)}\left(\int_{B_1}|\nabla w_0|^2\right)^{\frac{1}{2}}\\
&&-\int_0^{T'}\int_{B_1}\la A_u(\nabla u,\nabla u),w-w_0\ra t^{-\frac{1}{2}}dxdt.
\end{eqnarray*}
We will bound the third term from above. In the following calculation, energy non-increasing condition and (\ref{17thEqn}) guarantee that each quantity below is finite. Since $A_u(\nabla u,\nabla u)$ is perpendicular to $M$ at $u$ and $w=u-v$, we can apply Lemma \ref{EleGeo} to the $L^2$ inner product of $A_u(\nabla u,\nabla u)$ and $w$ on $B_1$. Note that $w-w_0=(u-u_0)-(v-v_0)\in H^1_0(B_1)$ for $a.e.$ $t$ fixed. Thus, we can apply the interior gradient estimate (\ref{12thEqn}) and Lemma \ref{2ndLem} to the $L^2$ inner product of $|\nabla u|^2$ and $|w-w_0|^2$ on $B_1$. Hence,
\begin{eqnarray*}
&&-\int_0^{T'}\int_{B_1}\la A_u(\nabla u,\nabla u),w-w_0\ra t^{-\frac{1}{2}}dxdt\\
&\le& \lam_1\int_0^{T'}\int_{B_1}|\nabla u|^2|w|^2t^{-\frac{1}{2}}dxdt
+\lam_2\int_0^{T'}\int_{B_1}|\nabla u|^2|w_0|t^{-\frac{1}{2}}dxdt\\
&\le& \lam_1\int_0^{T'}\int_{B_1}|\nabla u|^2|w-w_0|^2t^{-\frac{1}{2}}dxdt
+\lam_2\int_0^{T'}\int_{B_1}|\nabla u|^2(|w_0|^2+|w_0|)t^{-\frac{1}{2}}dxdt\\
\end{eqnarray*}
\begin{eqnarray*}
&\le&\lam_1C_1\eps_2\int_0^{T'}\int_{B_1}|w-w_0|^2t^{-\frac{3}{2}}dxdt
+4\lam_1C_1\eps_2\int_0^{T'}\int_{B_1}|\nabla w-\nabla w_0|^2t^{-\frac{1}{2}}dxdt\\
&&+\lam_2\int_0^{T'}\int_{B_1}|\nabla u|^2(|w_0|^2+|w_0|)t^{-\frac{1}{2}}dxdt,
\end{eqnarray*}
where $\lam_1=C\sup_M|A|$ and $\lam_2>0$ (changing from line to line in the computation above) depends only on M. The previous two inequalities and $32\lam_1C_1\eps_2<1$ give,
\begin{equation}
\label{20thEqn}
\begin{split}
&\quad\ \int_0^{T'}\int_{B_1}\la \nabla u,\nabla w\ra t^{-\frac{1}{2}}dxdt\\ 
&\le \int_0^{T'}\int_{B_1}\la -u_t,w-w_0\ra t^{-\frac{1}{2}}dxdt+\frac{1}{32}\int_0^{T'}\int_{B_1}
|w-w_0|^2t^{-\frac{3}{2}}dxdt\\
&\quad\ +\frac{1}{4}\int_0^{T'}\int_{B_1}|\nabla w|^2t^{-\frac{1}{2}}dxdt+N_1,
\end{split}
\end{equation} 
where
\begin{equation*}
\begin{split}
N_1&=\lam_2\int_0^{T'}\int_{B_1}
(|\nabla u|^2+|\nabla v|^2)(|w_0|^2+|w_0|)t^{-\frac{1}{2}}dxdt\\
&\quad\  +2\sqrt{2T'(E(u_0)+E(v_0))}\left(\int_{B_1}|\nabla w_0|^2\right)^{\frac{1}{2}}+\frac{\sqrt{T'}}{2}\int_{B_1}|\nabla w_0|^2dx.
\end{split}
\end{equation*}

Similarly,
\begin{equation}
\label{21thEqn}
\begin{split}
&\quad\ -\int_0^{T'}\int_{B_1}\la \nabla v,\nabla w\ra t^{-\frac{1}{2}}dxdt\\
&\le \int_0^{T'}\int_{B_1}\la v_t,w-w_0\ra t^{-\frac{1}{2}}dxdt+\frac{1}{32}\int_0^{T'}\int_{B_1}
|w-w_0|^2t^{-\frac{3}{2}}dxdt\\
&\quad\ +\frac{1}{4}\int_0^{T'}\int_{B_1}|\nabla w|^2t^{-\frac{1}{2}}dxdt+N_1.
\end{split}
\end{equation}

Thus, combining (\ref{20thEqn}) and (\ref{21thEqn}), we get
\begin{eqnarray*}
&&\int_0^{T'}\int_{B_1}|\nabla w|^2t^{-\frac{1}{2}}dxdt\\
&\le& -\int_0^{T'}\int_{B_1}\la w_t,w-w_0\ra t^{-\frac{1}{2}}dxdt+\frac{1}{16}\int_0^{T'}\int_{B_1}|w-w_0|^2t^{-\frac{3}{2}}dxdt\\
&&+\frac{1}{2}\int_0^{T'}\int_{B_1}|\nabla w|^2t^{-\frac{1}{2}}dxdt+2N_1\\
&\le&-\frac{1}{2\sqrt{T'}}\int_{\{T'\}\times B_1}|w-w_0|^2dx+\frac{1}{2}\int_0^{T'}\int_{B_1}|\nabla w|^2t^{-\frac{1}{2}}dxdt+2N_1,
\end{eqnarray*}
where the last inequality follows from integration by parts and (\ref{18thEqn}). Therefore,
\begin{equation}
\begin{split}
&\quad\ \int_0^{T'}\int_{B_1}|\nabla w|^2t^{-\frac{1}{2}}dxdt+\frac{1}{2\sqrt{T'}}\int_{\{T'\}\times B_1}|w|^2dx\\
&\le \frac{1}{\sqrt{T'}}\int_{B_1}|w_0|^2dx+4N_1,
\end{split}
\end{equation}
and $C_2=\lam_2$ in the lemma.
\end{proof}
Now, under the condition of Theorem \ref{1stThm}, that is, $u_0=v_0$, we have $N\equiv 0$ and thus $u(t,\cdot)=v(t,\cdot)$ in $L^2(B_1)$ for each $t\in [0,T']$. Therefore, it follows from an open and closed argument that  $u=v$ $a.e.$ on $[0,T]\times B_1$.

\section{Proof of Theorem \ref{2ndThm}}
Let $u$ be the weak solution of (\ref{HMHFEqn}) in $\cap_{T>0}H^1([0,T]\times B_1,M)$ with initial data $u_0\in H^1(B_1,M)$ and boundary data $\gamma=u_0|_{\partial B_1}$. First, define $E=\sup_{t\ge 0}E(u(t,\cdot))$. If $E<\eps_1$, then, by the similar argument used to obtain (\ref{12thEqn}), we get
\begin{equation}
\label{GradientSEnergy}
|\nabla u|^2(t_0,x_0)\le C_1E\max\{t_0^{-1}, (1-|x_0|)^{-2}\}\le C_1E\left[t_0^{-1}+(1-|x_0|)^{-2}\right],
\end{equation}
for $t_0>0$ and $x_0\in B_1$.

Second, we derive two estimates of the kinetic energy: one holds for $a.e.$ $t_0>0$ and the other holds for $t_0$ large enough. 
\begin{lem}
\label{4thLem}
There exists $\eps_3\in (0,\eps_1)$, depending only on $M$, so that: if $E\le\eps_3$, then for $a.e.$ $t_0>0$,
\begin{equation}
\label{15thEqn}
\int_{\{t_0\}\times B_1}|u_t|^2dx\le \frac{4}{t_0}\int_0^{t_0}\int_{B_1}|u_t|^2dxdt,
\end{equation}
and there exist $T_1>0$, $\alpha_1>0$ and $C_3>0$ such that for $a.e.$ $t_0\ge T_1$,
\begin{equation}
\label{16thEqn}
\int_{\{t_0\}\times B_1}|u_t|^2dx\le C_3\exp[-\alpha_1 (t-T_1)].
\end{equation}
\end{lem}
\begin{proof}
Assume that $E<\eps_1$. Let $0<h<h_0\ll 1$. We define the difference quotient $D^hu=\left(u(t+h,x)-u(t,x)\right)/h$. Note that 
\begin{equation}
\label{TriFact}
D^hu(t,\cdot)\in H^1_0\cap L^\infty(B_1)\ \text{for}\ a.e.\ t\ \text{fixed}.
\end{equation}
Thus, by equation (\ref{2ndEqn}), we get
\begin{eqnarray*}
&&\frac{d}{dt}\int_{\{t\}\times B_1}|D^hu|^2dx\\
&=&-2\int_{\{t\}\times B_1}|\nabla D^hu|^2dx+2h^{-1}\int_{\{t\}\times B_1}\la A_u(\nabla u,\nabla u),D^hu\ra dx\\
&&-2h^{-1}\int_{\{t\}\times B_1}\la A_{u(t+h,x)}(\nabla u,\nabla u),D^hu\ra dx\\
&\le& -2\int_{\{t\}\times B_1}|\nabla D^hu|^2dx
+\lam_1\int_{\{t\}\times B_1}|D^hu|^2\left(|\nabla u|^2(t,x)+|\nabla u|^2(t+h,x)\right)dx,
\end{eqnarray*}
where $\lam_1=C\sup_M|A|$, and we use (\ref{GeoFact}) and Lemma \ref{EleGeo} in the last inequality. For $t>0$, by  (\ref{GradientSEnergy}), (\ref{TriFact}) and Lemma \ref{2ndLem}, we have
\begin{eqnarray*}
&&\int_{\{t\}\times B_1}|D^hu|^2|\nabla u|^2dx\\
&\le&C_1E\int_{\{t\}\times B_1}|D^hu|^2\left[t^{-1}+(1-|x|^2)^{-2}\right]dx\\
&\le& 4C_1E\int_{\{t\}\times B_1}|\nabla D^hu|^2dx+C_1Et^{-1}\int_{\{t\}\times B_1}|D^hu|^2dx.
\end{eqnarray*}
If $8C_1\lam_1E\le 1$, then 
\begin{equation}
\label{RateEnergy}
\begin{split}
&\quad \frac{d}{dt}\int_{\{t\}\times B_1}|D^hu|^2dx\\
&\le -\int_{\{t\}\times B_1}|\nabla D^hu|^2dx+t^{-1}\int_{\{t\}\times B_1}|D^hu|^2dx\\
&\le -\left(C_s^{-1}-t^{-1}\right)\cdot\int_{\{t\}\times B_1}|D^hu|^2dx,
\end{split}
\end{equation}
where we apply the Sobolev inequality to $D^hu(t,\cdot)\in H^1_0(B_1)$ in the first inequality and $C_s>0$ is the Sobolev constant of $B_1$. 

Thus, integrating over $[t_0/2,t_0]$, (\ref{RateEnergy}) gives
\begin{equation}
\int_{\{t_0\}\times B_1}|D^hu|^2dx\le \frac{4}{t_0}\int_{0}^{t_0}\int_{B_1}|D^hu|^2dxdt.
\end{equation}
On the other hand, let $I=[2C_s,2C_s+1]$ and it is obvious that
\begin{equation}
\inf_{t\in I}\int_{\{t\}\times B_1}|D^hu|^2dx\le \int_I\int_{B_1}|D^hu|^2dxdt.
\end{equation}
Thus, for $t>T_1=2C_s+2$, (\ref{RateEnergy}) also implies that
\begin{equation}
\int_{\{t\}\times B_1}|D^hu|^2dx\le \int_I\int_{B_1}|D^hu|^2dxdt\cdot \exp[-C_s^{-1}(t-T_1)/2].
\end{equation}

Letting $h\To 0$ and by (\ref{8thEqn}), Lemma \ref{4thLem} follows with $\eps_3=\min\{\eps_1,C_1^{-1}\lam_1^{-1}/8\}$, $C_3=\int_I\int_{B_1}|u_t|^2dxdt$ and $\alpha_1=C_s^{-1}/2$.
\end{proof}
Finally, assume that $E<\eps_3$ and $32CC_1E\sup_M|A|<1$. In the calculation below, we first apply H\"{o}lder's inequality to the $L^2$ inner product of $u_t$ and $u(t_2,\cdot)-u(t_1,\cdot)$ on $B_1$. Meanwhile, we use (\ref{GeoFact}), Lemma \ref{EleGeo} and Lemma \ref{2ndLem} to bound the $L^2$ inner product of $A_{u(t_2,x)}(\nabla u,\nabla u)$ and $u(t_2,\cdot)-u(t_1,\cdot)$ on $B_1$. Next, we deduce the last inequality from Cauchy's inequality and the assumption on the upper bound of $E$. That is, for $a.e.$ $T_1<t_1<t_2$, 
\begin{eqnarray*}
&&\int_{B_1}\la \nabla u(t_2,x),\nabla u(t_2,x)-\nabla u(t_1,x)\ra dx\\
&=&\int_{B_1}\la -u_t(t_2,x)-A_{u(t_2,x)}(\nabla u,\nabla u),u(t_2,x)-u(t_1,x)\ra dx\\
&\le& \sqrt{C_s}\left(\int_{B_1}|u_t(t_2,x)|^2dx\right)
^{\frac{1}{2}}\left(\int_{B_1}|\nabla u(t_2,x)-\nabla u(t_1,x)|^2dx\right)^{\frac{1}{2}}\\
&&+4CC_1E\sup_M|A|\cdot\int_{B_1}|\nabla u(t_2,x)-\nabla u(t_1,x)|^2dx\\
&\le& 2C_s\int_{B_1}|u_t(t_2,x)|^2dx+\frac{1}{4}\int_{B_1}|\nabla u(t_2,x)-\nabla u(t_1,x)|^2dx.
\end{eqnarray*}
Similarly,
\begin{eqnarray*}
&&\int_{B_1}\la \nabla u(t_1,x),\nabla u(t_1,x)-\nabla u(t_2,x)\ra dx\\
&\le& 2C_s\int_{B_1}|u_t(t_1,x)|^2dx+\frac{1}{4}\int_{B_1}|\nabla u(t_1,x)-\nabla u(t_2,x)|^2dx.
\end{eqnarray*}
Summing the two inequalities above, we get
\begin{equation}
\int_{B_1}|\nabla u(t_1,x)-\nabla u(t_2,x)|^2dx\le 4C_s\int_{B_1}\left(|u_t(t_1,x)|^2+|u_t(t_2,x)|^2\right)dx.
\end{equation}
Therefore, Theorem \ref{2ndThm} follows immediately from Lemma \ref{4thLem} by choosing $\eps_0=\min\{\eps_3,C^{-1}C^{-1}_1\inf_M|A|^{-1}/32\}$.

\section{Example of harmonic map heat flow not in $V^T$}
In this section, we construct the unique weak solution $u\in\cap_{T>0}H^1([0,T]\times B_1,M)$ of (\ref{HMHFEqn}) starting with small energy initial data $u_0\in H^1\cap C^0(\bar{B}_1,M)$ and boundary data $\gamma=u_0|_{\partial B_1}$. In general, the weak solution $u$ is not in $V^T$.
\begin{prop}
\label{1stProp} There exists $\eps_4>0$, depending only on $M$, so that: given $u_0\in H^1\cap C^0(\bar{B}_1,M)$ with $E(u_0)<\eps_4$, there exists a unique weak solution in $\cap_{T>0}H^1([0,T]\times B_1,M)$ of (\ref{HMHFEqn}) whose energy is non-increasing. Moreover, for $0\le t_1<t_2$,
\begin{equation}
\label{22ndEqn}
\frac{1}{7}\int_{B_1}|\nabla u(t_2,x)-\nabla u(t_1,x)|^2dx\le\int_{B_1}|\nabla u(t_1,x)|^2dx-\int_{B_1}|\nabla u(t_2,x)|^2dx.
\end{equation}
\end{prop} 
\begin{rem}
\label{4thRem}
Recently, Colding and Minicozzi showed that the $H^1$ distance between a harmonic map and a $H^1$ map with the same boundary value can be controlled by their gap in energy, assuming energy is small; see Theorem 3.1 in \cite{CM2}. This is a key ingredient in the proof of the finite extinction of Ricci flow. Combining with Theorem \ref{2ndThm} and Remark \ref{3rdRem}, our estimate (\ref{22ndEqn}) can be viewed as a parabolic version of their theorem. 
\end{rem}
\begin{proof}
First, we may approximate $u_0$ by a sequence of maps $u_{m0}\in C^\infty(\bar{B}_1,M)$ in $H^1\cap C^0$ topology.\footnote{Let $v_{m0}\in C^\infty(\bar{B}_1,\Real^N)$ be the global approximations of $u_0$, constructed in Theorem 3 on page 252 of \cite{E}. Then $u_{m0}$ could be the nearest point projection (onto $M$) of $v_{m0}$.} By Theorem 1.1 in \cite{Ch}, there exists $\delta_1\in(0,\eps_2)$, depending only on $M$, so that: if $E(u_{m0})<\delta_1$, then the weak solution $u_m\in \cap_{T>0}W^{1,2}_p\cap C^{1+(\mu/2),2+\mu}((0,T)\times \bar{B}_1,M)$\footnote{$w\in W^{1,2}_p((0,T)\times \bar{B}_1,M)$ means that $u$, $|\nabla u|$, $|\nabla^2 u|$ and $|u_t|$ are in $L^p((0,T)\times \bar{B}_1)$, and $u\in M$ for $a.e.$ $(t,x)\in (0,T)\times \bar{B}_1$.} of (\ref{HMHFEqn}) exists, $\forall 0<\mu<1,p>4/(1-\mu)$, and for $0\le  t_1<t_2$,
\begin{equation}
\label{23rdEqn}
\int_{t_1}^{t_2}\int_{B_1}|\partial _tu_m|^2dxdt=E(u_m(t_1,\cdot))-E(u_m(t_2,\cdot)).
\end{equation}
If $E(u_0)<\delta_1/2$, then, by Lemma \ref{3rdLem} and a diagonalization argument, there exists a subsequence (relabeled) of $u_m$ satisfying that for $\forall T>0$, $\partial _tu_m\rightharpoonup \partial_tu$ weakly in $L^2([0,T]\times B_1)$, $u_m\To u$ and $\nabla u_m\To \nabla  u$ strongly in $L^2([0,T]\times B_1)$. Note that the boundary data $\gamma_m\To \gamma$ in $H^{1/2}\cap C^0(\partial B_1,M)$. Therefore, $u\in\cap_{T>0}H^1([0,T]\times B_1)$ is the weak solution of (\ref{HMHFEqn}) with initial data $u_0$ and boundary data $\gamma=u_0|_{\partial B_1}$. Moreover, $E(u(t,\cdot))\le E(u_0)$ for $a.e.$ $t$ and there exists a zero measure set $I_1\subseteq(0,\infty)$ so that: if $t_1,t_2\in I_1^c$ and $t_1<t_2$, then
\begin{equation}
\label{EnergyIneqn}
\int_{t_1}^{t_2}\int_{B_1}|u_t|^2dxdt\le E(u(t_1,\cdot))-E(u(t_2,\cdot)).
\end{equation}

Second, we will show that inequality (\ref{EnergyIneqn}) is actually equality for $a.e.$ $0<t_1<t_2$. Let $0<h<h_0\ll 1$. Define $D^hu(t,x)=(u(t+h,x)-u(t,x))/h$. Thus, for $0<t_1<t_2$,
\begin{eqnarray*}
&&\int_{t_1}^{t_2}\int_{B_1}|u_t|^2dxdt=\lim_{h\To 0}\int_{t_1}^{t_2}\int_{B_1}\la u_t,D^hu\ra dxdt\\
&=&\lim_{h\To 0}\int_{t_1}^{t_2}\int_{B_1}-\la \nabla u,\nabla D^hu\ra dxdt-\int_{t_1}^{t_2}\int_{B_1}\la A_u(\nabla u,\nabla u),D^hu\ra dxdt.
\end{eqnarray*}
We will bound the second term. In the following calculation, we use (\ref{GeoFact}) and Lemma \ref{EleGeo} in the first inequality, apply the gradient estimate (\ref{GradientSEnergy}) in the second inequality, and use (\ref{TriFact}) and Lemma \ref{2ndLem} in the last inequality. Thus,
\begin{eqnarray*}
&&|\int_{t_1}^{t_2}\int_{B_1}\la A_u(\nabla u,\nabla u),D^hu\ra dxdt|\\
&\le& hC\sup_M|A|\int_{t_1}^{t_2}\int_{B_1}|\nabla u|^2|D^hu|^2dxdt\\
&\le& hCC_1E(u_0)\sup_M|A|\int_{t_1}^{t_2}\int_{B_1}[t^{-1}+(1-|x|)^{-2}]|D^hu|^2dxdt\\
&\le&t_1^{-1}hCC_1E(u_0)\sup_M|A|\int_{t_1}^{t_2}\int_{B_1}|D^hu|^2dxdt\\
&&+4hCC_1E(u_0)\sup_M|A|\int_{t_1}^{t_2}\int_{B_1}|\nabla D^hu|^2dxdt,
\end{eqnarray*}
For the first term,
\begin{align*}
&\quad\ \int_{t_1}^{t_2}\int_{B_1}-\la \nabla u,\nabla D^hu \ra dxdt\\
&=\frac{1}{2h}\int_{t_1}^{t_2}\int_{B_1}\left(|\nabla u(t,x)|^2-|\nabla u(t+h,x)|^2\right)dxdt+\frac{h}{2}\int_{t_1}^{t_2}\int_{B_1}|\nabla D^hu|^2dxdt\\
&=\frac{1}{2h}\int_{t_1}^{t_1+h}\int_{B_1}|\nabla u|^2dxdt-\frac{1}{2h}\int_{t_2}^{t_2+h}\int_{B_1}|\nabla u|^2dxdt+\frac{h}{2}\int_{t_1}^{t_2}\int_{B_1}|\nabla D^hu|^2dxdt.
\end{align*}
If $E(u_0)<C^{-1}C_1^{-1}\inf_M|A|^{-1}/8$, then
\begin{eqnarray*}
&&\int_{t_1}^{t_2}\int_{B_1}|u_t|^2dxdt\\
&\ge&\lim_{h\To 0}\frac{1}{2h}\int_{t_1}^{t_1+h}\int_{B_1}|\nabla u|^2dxdt-\frac{1}{2h}\int_{t_2}^{t_2+h}\int_{B_1}|\nabla u|^2dxdt-\frac{h}{8t_1}\int_{t_1}^{t_2}\int_{B_1}|D^hu|^2dxdt\\
&\ge&\lim_{h\To 0}\frac{1}{2h}\int_{t_1}^{t_1+h}\int_{B_1}|\nabla u|^2dxdt-\frac{1}{2h}\int_{t_2}^{t_2+h}\int_{B_1}|\nabla u|^2dxdt.
\end{eqnarray*}
Define $f(t)=\int_{B_1}|\nabla u(t,x)|^2dx$ and $F(t)=\int_0^tf(s)ds$. Since $f\in L^1([0,T])$ for $\forall T\in (0,\infty)$, $F'(t)=f(t)$ for $a.e.$ $t\in [0,T]$. Hence, there exists a zero measure set $I_2\subseteq (0,\infty)$ so that: if $t\in I_2^c$, then
\begin{equation}
\label{CalculusThm}
\lim_{h\To 0}\frac{1}{h}\int_t^{t+h}\int_{B_1}|\nabla u|^2dxdt=\int_{B_1}|\nabla u(t,x)|^2dx.
\end{equation}
Therefore, if $t_1,t_2\in I_2^c$ and  $0<t_1<t_2$, then
\begin{equation}
\label{EnergyIneqn2}
\int_{t_1}^{t_2}\int_{B_1}|u_t|^2dxdt\ge E(u(t_1,\cdot))-E(u(t_2,\cdot)).
\end{equation}
Combining with inequality (\ref{EnergyIneqn}), we get
\begin{equation}
\label{EnergyEqn}
\int_{t_1}^{t_2}\int_{B_1}|u_t|^2dxdt=E(u(t_1,\cdot))-E(u(t_2,\cdot)),
\end{equation}
if $t_1,t_2\in (I_1\cup I_2)^c$ and $0<t_1<t_2$.

Third, assume that $E(u_0)<\min\{\eps_3/2,\delta_1/2,C^{-1}C_1^{-1}\inf_M|A|^{-1}/32\}$. Since $u_{m0}\To u_0$ in $H^1(B_1,M)$ topology, we may also assume that $E(u_{m0})\le 2E(u_0)$. Thus, for $0\le t_1<t_2$, using (\ref{GradientSEnergy}), Lemmas \ref{EleGeo}, \ref{2ndLem} and \ref{4thLem},\footnote{Note that Lemma \ref{4thLem} holds for every $t_0>0$ for $u_m$, since $u_m\in C^{1+(\mu/2),2+\mu}((0,\infty)\times \bar{B}_1,M)$ and $\norm{u_t(t,\cdot)}_{L^2(B_1)}$ is continuous in $t$.} we estimate
\begin{eqnarray*}
&&\int_{B_1}\la \nabla u_m(t_2,x),\nabla u_m(t_2,x)-\nabla u_m(t_1,x)\ra dx\\
&=&\int_{B_1}\la -\partial_tu_m(t_2,x)-A_{u_m(t_2,x)}(\nabla u_m,\nabla u_m),u_m(t_2,x)-u_m(t_1,x)\ra dx\\
&\le& \left(\int_{B_1}|\partial_tu_m(t_2,x)|^2dx\right)^{\frac{1}{2}}
\left(\int_{B_1}|u_m(t_1,x)-u_m(t_2,x)|^2dx\right)^{\frac{1}{2}}\\
&&+C\sup_M|A|\int_{B_1}|\nabla u_m(t_2,x)|^2|u_m(t_1,x)-u_m(t_2,x)|^2dx\\
&\le& \sqrt{t_2-t_1}\cdot\left(\int_{B_1}|\partial_t u_m(t_2,x)|^2\right)^{\frac{1}{2}}\left(\int_{t_1}^{t_2}\int_{B_1}|\partial_t
u_m|^2dxdt\right)^{\frac{1}{2}}\\
&&+CC_1E(u_m(t_1,\cdot))\sup_M|A|\int_{B_1}(t_2-t_1)^{-1}|u_m(t_1,x)-u_m(t_2,x)|^2dx\\
&&+CC_1E(u_m(t_1,\cdot))\sup_M|A|\int_{B_1}\frac{|u_m(t_1,x)-u_m(t_2,x)|^2}
{(1-\sqrt{x_1^2+x_2^2})^2}dx\\
&\le&\frac{33}{16}\int_{t_1}^{t_2}\int_{B_1}|\partial_t u_m|^2dxdt+\frac{1}{4}\int_{B_1}|\nabla u_m(t_2,x)-\nabla u_m(t_1,x)|^2dx,
\end{eqnarray*}
where we use the fact that the energy of $u_m$ is non-increasing in the last inequality. Thus, 
\begin{eqnarray*}
&&\int_{B_1}|\nabla u_m(t_1,x)|^2dx-\int_{B_1}|\nabla u_m(t_2,x)|^2dx\\
&=&\int_{B_1}|\nabla u_m(t_1,x)-\nabla u_m(t_2,x)|^2dx\\
&&+2\int_{B_1}\la \nabla u_m(t_2,x),\nabla u_m(t_1,x)-\nabla u_m(t_2,x)\ra dx\\
&\ge&\frac{1}{2}\int_{B_1}|\nabla u_m(t_1,x)-\nabla u_m(t_2,x)|^2dx-5\int_{t_1}^{t_2}\int_{B_1}|\partial_tu_m|^2dxdt.
\end{eqnarray*}
Hence, it follows from (\ref{23rdEqn}) that
\begin{equation}
\label{24thEqn}
\frac{1}{7}\int_{B_1}|\nabla u_m(t_1,x)-\nabla u_m(t_2,x)|^2dx\le\int_{B_1}|\nabla u_m(t_1,x)|^2dx-\int_{B_1}|\nabla u_m(t_2,x)|^2dx.
\end{equation}
By Lemma \ref{3rdLem}, there exists a zero measure set $I_3\subseteq (0,\infty)$ so that: if $t\in I_3^c$, then there exists a subsequence (relabeled) of $u_m$ such that $u_m(t,\cdot)\To u(t,\cdot)$ in $H^1(B_1,M)$ topology. Hence, if $0<t_1<t_2$ and $t_1,t_2\in (I_1\cup I_2\cup I_3)^c$, then
\begin{eqnarray}
&&\frac{1}{7}\int_{B_1}|\nabla u(t_2,x)-\nabla u(t_1,x)|^2dx\le\int_{B_1}|\nabla u(t_1,x)|^2dx-\int_{B_1}|\nabla u(t_2,x)|^2dx\\
&&\frac{1}{2}\int_{B_1}|\nabla u(t_1,x)|^2dx-\frac{1}{2}\int_{B_1}|\nabla u(t_2,x)|^2dx=\int_{t_1}^{t_2}\int_{B_1}|u_t|^2dxdt.
\end{eqnarray}
We can modify the definition of $u$ on $(I_1\cup I_2\cup I_3)\times B_1$ by taking limits. Therefore, the modified map solves (\ref{HMHFEqn}) satisfying that the energy is non-increasing and the uniqueness for weak solutions of the harmonic map heat flow follows from Theorem \ref{1stThm}.
\end{proof}

\appendix
\section{}
\begin{prop}
\label{IntrinDist}
There exists $\eps_5>0$, depending on $M$, so that: if $|x-y|<\eps_5$, then $\dist_{M}(x,y)<2|x-y|$, where $\dist_M(x,y)$ is the intrinsic distance between $x$ and $y$ on $M$.
\end{prop}
\begin{proof}
If not, then there exists a sequence of $(x_j,y_j)\in M\times M$ such that $|x_j-y_j|\To 0$ but $\dist_{M}(x_j,y_j)\ge 2|x_j-y_j|$. Since $M$ is compact, there exist $x_0\in M$ and a subsequence (relabeled) of $(x_j,y_j)$ satisfying that $x_j\To x_0$ and $y_j\To x_0$. There exists $0<\delta_1<\sup_M|A|/4$ such that the geodesic ball $B^M_{\delta_1}(x_0)$ centered at $x_0$ with radius $\delta_1$ is strictly geodesically convex. If $j$ is sufficiently large, then $x_j$ and $y_j$ are in $B^M_{\delta_1}(x_0)$. Let $l_j$ be the geodesic distance between $x_j$ and $y_j$, and $\gamma_j:[0,l_j]\To B_{\delta_1}^M(x_0)$ be the unit speed minimizing geodesic joining $x_j$ and $y_j$. Thus,
\begin{eqnarray*}
|y_j-x_j|^2&=&\int_0^{l_j}2\la \gamma_j(s)-x_j,\gamma_j'(s)\ra ds\\
&=&\int_0^{l_j}\int_0^s\left(2|\gamma_j'(\tau)|^2+2\la \gamma_j(\tau)-x_j,\gamma ''(\tau)\ra\right)d\tau ds\\
&\ge&\int_0^{l_j}\int_0^s2\left(1-2\delta_1\sup_M|A|\right)d\tau ds\\
&\ge&\frac{l_j^2}{2}.
\end{eqnarray*}
Therefore, $\dist_M(x_j,y_j)\le \sqrt{2}|x_j-y_j|$ and this is a contradiction.
\end{proof}
\begin{proof}
(of Lemma \ref{EleGeo}) If $|x-y|\ge\eps_5$, then $|(x-y)^{\perp}|/|x-y|^2\le\eps_5^{-1}$. Otherwise, let $\gamma:[0,l]\To M$ be the minimizing geodesic joining $y$ to $x$ with length $l\le 2|x-y|$.
\begin{eqnarray*}
|(x-y)^{\perp}|&=&\int_0^l\la \gamma '(s),V\ra ds=\int_0^l\int_0^s\la \gamma ''(\tau),V\ra d\tau ds\\
&\le&\sup_M|A|\cdot\frac{l^2}{2}\le 2\sup_M|A|\cdot|x-y|^2,
\end{eqnarray*}
where $V=(x-y)^{\perp}/|(x-y)^{\perp}|$. Therefore, Lemma \ref{EleGeo} follows immediately with $C=\max\{\eps_5^{-1},2\sup_M|A|\}$.
\end{proof}

\bibliographystyle{plain}

\end{document}